\documentclass[12pt]{amsart}

\usepackage{CJKutf8}
\usepackage{pgf}
\usepackage{graphicx}
\usepackage{amsmath, amsthm, amssymb, amscd}

\newcommand{\eqd}[1]{\approx_{\delta_{#1}}}

\newcommand{\eqe}{\approx_{\epsilon}}

\newcommand\Z{{\mathbb Z}}
\newcommand\N{{\mathbb N}}
\newcommand\R{{\mathbb R}}
\newcommand\C{{\mathbb C}}

\newcommand\Q{{\mathbb Q}}

\newcommand\bp{\begin{proof}}
\newcommand\ep{\end{proof}}

\newcommand\bprop{\begin{prop}}
\newcommand\eprop{\end{prop}}

\newcommand\bt{\begin{thm}}
\newcommand\et{\end{thm}}

\newcommand\bc{\begin{cor}}
\newcommand\ec{\end{cor}}

\newcommand\ba{\begin{aligned}}
\newcommand\ea{\end{aligned}}

\newcommand\bl{\begin{lem}}
\newcommand\el{\end{lem}}

\newcommand\bi{\begin{itemize}}
\newcommand\ei{\end{itemize}}

\newcommand\br{\begin{rem}}
\newcommand\er{\end{rem}}

\newcommand\bd{\begin{defn}}
\newcommand\ed{\end{defn}}

\newcommand\Nn{{\mathcal N}}

\newcommand\Qq{{\mathcal Q}}
\newcommand\Zz{{\mathcal Z}}

\DeclareMathOperator{\End}{End}
\DeclareMathOperator{\Aut}{Aut}

\DeclareMathOperator{\Ad}{Ad}
\DeclareMathOperator{\diag}{diag}

\DeclareMathOperator{\id}{id}

\newtheorem{thm}{Theorem}[section]

\newtheorem*{thm*}{Theorem}

\newtheorem{prop}[thm]{Proposition}
\newtheorem{lem}[thm]{Lemma}
\newtheorem{cor}[thm]{Corollary}

\newtheorem*{q*}{Question}

\theoremstyle{remark}
\newtheorem{rem}[thm]{Remark}
\newtheorem*{rem*}{Remark}

\theoremstyle{definition}
\newtheorem{defn}[thm]{Definition}

\newtheorem{eg}[thm]{Example}
\address{\tiny Department of Mathematics, University of Oregon, Eugene, OR 97403-1222, USA}
\email{msun@uoregon.edu}
\author{Michael Sun}
\title[Strongly outer product type actions]{Strongly outer product type actions}
\begin{document}
\begin{abstract}We show that for any countable discrete maximally almost periodic group $G$ and any UHF algebra $A$, there exists a strongly outer product type action $\alpha$ of $G$ on $A$. When $G$ is also elementary amenable, all such crossed products $A\rtimes_{\alpha}G$ are simple nuclear tracially approximately finite dimensional (TAF) with a unique tracial state and satisfy the Universal Coefficient Theorem (UCT). We also show the existence of countable discrete almost abelian group actions on the universal UHF algebra with a certain Rokhlin property.
\end{abstract}
\maketitle
\section*{Introduction}
 One way to define a UHF algebra is to start with a sequence $(n_l)_{l\in\N}$ of integers greater than 1 and then associate to it the $C^*$-algebra $M_{(n_l)_{l\in\N}}$ using an infinite tensor product. That is
$$M_{(n_l)_{l\in\N}}=M_{n_1}\otimes M_{n_2}\otimes\dots \otimes M_{n_l}\otimes\dots.$$

 
We will study actions of discrete groups $G$ on UHF algebras as well as study their the crossed product $C^*$-algebras $M_{(n_l)_{l\in\N}}\rtimes G$. We are particularly interested in finding examples of group actions whose crossed products retain certain key properties of $M_{(n_l)_{l\in\N}}$ such as having a unique tracial state.
Given our definition of UHF algebras, a natural starting point is to look at those group actions that preserve some tensor product decomposition. Furthermore we will for convenience look at those actions that are \emph{inner} on each factor. So that for each $l\in\N$, we have a group homomorphism
$$G\to U(M_{n_l})\to \Aut M_{n_l}.$$
Putting this together we get whats called a \emph{product type action} \cite{HR}:
$$G\to \prod_{l=1}^{\infty}U(M_{n_l})\to\Aut M_{(n_l)_{l\in\N}}.$$
To ensure that the crossed product is separable, we require our group $G$ to be countable. To guarantee that $M_{(n_l)_{l\in\N}}\rtimes G$ is simple, Kishimoto \cite{Kish1} taught us that it suffices to require $\alpha_g$ is not an inner automorphism for all $g\neq1$. In particular such a requirement would mean we have an embedding of our group $G$
$$G\hookrightarrow \prod_{l=1}^{\infty}U(M_{n_l})$$
making it \emph{maximally almost periodic}. We will require even more of our group actions, in that we want the automorphisms to remain outer even after extending them to the weak closure of $M_{(n_l)_{l\in\N}}$ in its trace representation. Such a requirement is called \emph{strongly outer} and Kishimoto \cite{kish3} has shown that this condition is essentially equivalent to $M_{(n_l)_{l\in\N}}\rtimes G$ having a unique tracial state. So we ask:
\begin{q*} For any countable discrete maximally almost periodic group $G$ and any given UHF algebra $A$, does there always exist a strongly outer product type action $\alpha$ of $G$ on $A$? What can be said about $A\rtimes_{\alpha}G$?
\end{q*}
We answer the first part in Section 3 as Theorem \ref{tRokA}:
\begin{thm*} If $G$ is a countable discrete maximally almost periodic group and $A$ is any UHF algebra, then there is a strongly outer product type action of $G$ on $A$. 
\end{thm*}

Some examples for abelian groups are given in Section 4.

For the second part of the question about $A\rtimes_{\alpha}G$, we examine recent work of Matui and Sato \cite{MS2} \cite{MS3}, from which much of this work was inspired. One of the things that they showed is that for many amenable groups, including all \emph{elementary amenable} groups, and for a large class of $C^*$-algebras, including the UHF algebras, that strongly outer actions have some sort of \emph{tracial Rokhlin property}. This can then be used to show that the crossed product is \emph{Jiang-Su absorbing}, which in our case implies $A\rtimes_{\alpha}G$ is \emph{tracially approximately finite dimensional} or \emph{tracial rank zero} in the sense of Lin \cite{LinTAF}, which is a concise characterization of a large class of $C^*$-algebras satisfying Elliott's classification conjecture. We show in Section 6:

\begin{thm*}\label{cbtrz}Suppose $G$ is a countable discrete maximally almost periodic elementary amenable group, $A$ is any UHF algebra and $\alpha$ is a strongly outer product type action of $G$ on $A$. Then $A\rtimes_{\alpha}G$ is unital simple separable nuclear with tracial rank zero, satisfies the Universal Coefficient Theorem and has a unique tracial state. 
\end{thm*}

As an application of the existence theorem, we examine the existence problem for product type actions with the Rokhlin property in Section 5. Of course, we cannot hope for existence in general since we are already hindered by finite group actions requiring matrix sizes to be divisible by the order of the group. Nonetheless, the existence can be examined on the universal UHF algebra $\Qq$ where such divisibility constraints are satisfied and we relate it to our existence theorem using a simple procedure we call ``cutting-down''. In the absence of a general definition of the Rokhlin property we use a temporary substitute in our special case and prove the following, which appears as Theorem \ref{RokQ}:
\begin{thm*} If $G$ is a countable discrete almost abelian group, then there is a product type action of $G$ on $\Qq$ with the ``Rokhlin property''.
\end{thm*}

When $G$ is abelian, the extra strength of this property will be enough to determine $A\rtimes_{\alpha}G$ up to isomorphism independently of $\alpha$. This will appear in another article.


\subsection*{Acknoledgements} I would like to thank H. Lin, N. C. Phillips and Y. Sato. This is based on part of the author's thesis \cite[Ch III-V]{thesis}.

\section{Preliminaries}
\subsection*{Notation} 
We write $a\approx_{\epsilon}b$ to stand for $\|a-b\|<\epsilon$.

\subsection{Discrete groups}
 A discrete group $G$ is said to be \emph{maximally almost periodic} if for every $g\in G\setminus\{1\}$ there exists $n\in\N$ and a group homomorphism $\varphi_n: G\to U(M_n)$ such that $\varphi_n(g)\neq1$. 

\bl\label{map} If $G$ is a countable discrete maximally almost periodic group, then there is an embedding
$$G\hookrightarrow \prod_{n=1}^{\infty}U(M_n)$$
such that the image of $G$ intersects $\prod_{n=1}^{\infty} \C1_n$ trivially.
\el
\bp
For each $g\in G\setminus\{1\}$, let $\varphi[g]:G\to U(M_{n(g)})$ be non-trivial. Define a map $\psi[g]:G\to U(M_{n(g)+1})$ by $\psi[g]=\diag(1,\varphi[g])$
for all $g\in G\setminus\{1\}$. Now $\prod_g\psi[g]$ leads to the desired embedding. \ep

A discrete group is said to be \emph{almost abelian} if it has an abelian subgroup of finite index.

\begin{lem}\label{almostn} Suppose that $G$ is a discrete group with an abelian subgroup $H$ of finite index. Then the subgroup 
$$N=\bigcap_{g\in G}gHg^{-1}$$
is normal and abelian with finite index. In particular, every almost abelian group is elementary amenable.
\end{lem}
\begin{proof} The subgroup $N$ is abelian because it is a subgroup of $H$. It is a normal subgroup of $G$ by construction. Now $N$ is exactly the kernel of the action of $G$ on $G/H$, which is a finite set and hence $G/N$ embeds into the finite symmetric group $S_{[G:H]}$.
\end{proof}

\bl \label{induction}Suppose $G$ is a discrete group, $H$ is a subgroup of $G$ with finite index $k$ and let
$$N=\bigcap_{g\in G}gHg^{-1}.$$
If $\rho_H$ is a unitary representation of $H$ on $\C^n$ and $\rho_H^G$ its induced representation of $G$ on $\C^{nk}$, then $\rho_H^G|_N=\rho_H|_N\otimes\id_{M_k}$.
\el
\bp
Let $\{g_1,\dots, g_k\}$ be a set of coset representatives for $H$ and recall the induced representation can be written as
$$\C^{nk}=g_1\C^n\oplus g_2\C^n\oplus\dots\oplus g_k\C^n,$$
with the obvious $G$ action. Since $N$ stabilizes each of the cosets $g_iH$, it preserves this decomposition. On each factor the representations are conjugate.
\ep
\bprop \label{everyab}Every countable discrete abelian group is isomorphic to a subgroup of $\bigoplus_{n=1}^{\infty}(\Q\oplus\Q/\Z)$.
\eprop
\bp
We sketch a proof. First every abelian group is a subgroup of a divisible group. Every divisible group is a subgroup of a direct sum of copies of $\Q$ and $\Q/\Z$. Countable groups will only need countably many summands. 
\ep

\subsection{Product type actions on UHF algebras}
Let $\Qq$ be the universal UHF algebra conventionally given as the $C^*$-direct limit
$$\varinjlim \left(M_{n!},\phi_n: a\mapsto\diag(a,\dots,a)\right).$$
To represent this as an infinite tensor product first write
$$M_{n!}=M_1\otimes\dots\otimes M_n\quad\text{and}\quad\phi_n=\id_{n!}\otimes1_{n+1}$$
so that this limit gives meaning to the expression for $\Qq$ as
$$\bigotimes_{n=1}^{\infty}M_n=M_1\otimes M_2\otimes M_3\otimes\dots.$$
In a similar way we can define for any sequence $(n_l)_{l\in\N}$ of integers each at least 2, its associated UHF algebra $M_{(n_l)_{l\in\N}}$.
So we have in this notation $\Qq=M_{(n)_{n\in\N}}$.
We say that sequences $(n_l)_{l\in\N}$ and $(m_l)_{l\in\N}$ are of the same \emph{type} if  $M_{(n_l)_{l\in\N}}\cong M_{(m_l)_{l\in\N}}.$.
We will often write for a constant sequence $(n)_{l\in\N}$ as $M_{n^{\infty}}$.
For a sequence $(n_l)_{l\in\N}$ let
$$\Ad_{ (n_l)_{l\in\N}}:\prod_{l=1}^{\infty}U(M_{n_l})\to \Aut M_{(n_l)_{l\in\N}}$$
be the map defined for $u_{n_l}\in U(M_{n_l})$ by
$$(u_{n_l})_{l=1}^{\infty}\mapsto \bigotimes_{l=1}^{\infty}\Ad u_{n_l}$$
on the algebraic direct limit and then take the extension to the $C^*$-direct limit. Usually we will write $\Ad=\Ad_{ (n_l)_{l\in\N}}$. We consider an action of $G$ to be a \emph{product type action} if it is represented by a group homomorphism that has the following factorisation
$$G\to \prod_{l=1}^{\infty} U(M_{n_l})\stackrel{\Ad}{\rightarrow}\Aut M_{(n_l)_{l\in\N}}$$
for some strictly positive sequence $(n_l)_{l\in\N}$.

We now state a fundamental way to manipulate product type actions.
\bl\label{regroup} Let $\Nn=\{\mathfrak{n}_1,\mathfrak{n}_2,\mathfrak{n}_3,\dots\}$ be a countable pairwise disjoint cover of $\N$ such that each $\mathfrak{n}_j$ is ordered and let $(n_l)_{l\in\N}$ be a sequence of positive integers. Then there is a canonical isomorphism
$$ \bigotimes_{j=1}^{\infty}\bigotimes_{l\in \mathfrak{n}_j}M_{n_l}\to M_{(n_l)_{l\in\N}}.$$
Furthermore, if there was an action of $G$ preserving the the decomposition on the left and inner on each factor, then there is a canonical action of $G$ preserving the decomposition on the right and inner on each factor. \qed
\el
When the covers are finite and sequential, this amounts to \emph{regrouping} the factors. When the covers are singletons, this gives a \emph{reordering} of the factors. It can also help us to reindex a double infinite tensor product into a single infinite tensor product. 

\subsection{Crossed products by product type actions}
For any unital $C^*$-algebra $A$, any discrete group $G$ and an action $\alpha$ of $G$ on $A$, define the crossed product $A\rtimes_{\alpha}G$ as the $C^*$-algebra with the presentation
$$A\rtimes_{\alpha}G=\langle a, u_g\,|\,a\in A, g\in G, \alpha_g(a)=u_gau_g^*\rangle.$$

Recall the (full) group $C^*$-algebra is given by $C^*(G)=\C\rtimes_{\id}G$.


\begin{thm}\label{productcp}Let $A$ be a UHF algebra, let $G$ be a discrete group and let $\alpha$ be a product type action of $G$ on $A$ preserving the decomposition $A\cong M_{(n_l)_{l\in\N}}$ for some sequence $(n_l)_{l\in\N}$. Let $g_m$ be the image of $g$ in $U(M_{n_m})$ and suppose
$$\Phi_m: M_{(n_l)_{l\leq m}}\otimes C^*(G)\to M_{(n_l)_{l\leq m+1}}\otimes C^*(G)$$
is the $*$-homomorphism defined on generators by
$$x\otimes u_g\mapsto ((x\otimes1_{n_{m+1}})g_{n+1})\otimes u_g.$$
Then 
$$A\rtimes_{\alpha}G\cong \varinjlim\left(\Phi_m,M_{(n_l)_{l\leq m}}\otimes C^*(G)\right).$$
\end{thm}

\begin{proof}Since the actions are of product type, the crossed product is a limit of crossed products of full matrix algebras by inner actions. Crossed products from inner actions are tensor products in a natural way. To get the explicit connecting maps one simply follows 	the definitions of the isomorphisms (see \cite{thesis}).														
\end{proof}

\bc\label{cpqduct} Suppose $G$ is a countable discrete amenable maximally almost periodic group and suppose $A$ is any UHF algebra. If $\alpha$ is a product type action of $G$ on $A$, then $A\rtimes_{\alpha}G$ is nuclear quasidiagonal and satisfies the UCT. Furthermore, if $G$ is almost abelian then $A\rtimes_{\alpha}G$ is locally type I.
\ec
\bp By Theorem \ref{cpqduct} it suffices to show that $C^*(G)$ has the required properties. Amenability of $C^*(G)$ is \cite{Gui}. That $C^*(G)$ is quasidiagonal is discussed in the introduction of \cite{dadarqd} among other places. Also, $C^*(G)$ satisfies the UCT by \cite[Proposition 6.1]{WELP}. That $C^*(G)$ is type I when $G$ is almost abelian follows from \cite[Theorem 1 and Theorem 3]{Kap2}.
\ep
\subsection{The tracial Rokhlin property}

\bd Let $A$ be a unital $C^*$-algebra and $\alpha\in\Aut A$. We say $\alpha$ has the tracial Rokhlin property, if for every $n'\in\N$, there exists $N'> n'$ such that for every $\epsilon>0$ and every finite set $\{a_1,\dots, a_n\}$ in $A$, there exist mutually orthogonal projections $p_1,p_2,\dots ,p_{N'}$ such that 
\bi
\item $[p_i,a_j]\eqe0$ for $1\leq i\leq N'$ and $1\leq j\leq n$,\\
\item $\alpha(p_i)\eqe p_{i+1}$ for $1\leq i\leq N'-1$,\\
\item $\tau(p)\eqe 1$.
\ei
If this is true with $p=1$, we say that $\alpha$ has the Rokhlin property. 
\ed

For group actions, the Rokhlin property for finite groups is nicest
\begin{eg}\label{finite} Let $G$ be a finite group, then the infinite tensor product of regular representations gives an action of $G$ on $M_{|G|^{\infty}}$ with the Rokhlin property.
\end{eg}

\subsection{Strong outerness}

Let $A$ be a UHF algebra with tracial state $\tau$. Write $(\pi_{\tau}, H_{\tau})$ for the Gelfand-Naimark-Segal (GNS) representation for $\tau$. Then $\alpha\in\Aut A$ is \emph{weakly inner} if when extended to an automorphism of $\pi_{\tau}(A)''$, the weak closure of the $\pi_{\tau}(A)$ in $B(H_{\tau})$, there is an unitary $u\in\pi_{\tau}(A)''$ such that $\alpha=\Ad u$. An action is called \emph{strongly outer} if $\alpha_g$ is not weakly inner for all $g\in G\setminus\{1\}$.

In the case  an automorphism $\alpha$ is weakly inner we have for any $\tau$
$$\pi_{\tau}(\alpha(a))=u\pi_{\tau}(a)u^*.$$
This means that there is a representation $\pi$ extending $\pi_{\tau}$ to $A\rtimes\Z$.
We also have the following trace on on $A\rtimes\Z$ extending $\tau$:
 $$\tau_{A\rtimes\Z}(x)=\langle\pi(x)\widetilde{1_A}, \widetilde{1_A}\rangle,$$
where $\widetilde{1_A}\in H_{\tau}$ is the cyclic vector of $H_{\tau}$.

Define  for $a\in A$, $\|a\|_2=\tau(a^*a)^{1/2}$. The next lemma is adapted from Matui-Sato \cite{MS2} and we thank Y. Sato for communicating their proof.

\begin{lem}\label{ot} Let $A_n$ be a sequence of unital simple nuclear $C^*$-algebras with unique tracial state and let $\alpha_n\in\Aut(A_n)$. For $A=\bigotimes_{n=1}^{\infty}A_n$ define $\alpha\in\Aut A$ by
$$\alpha=\bigotimes_{n=1}^{\infty}\alpha_n.$$
If there is a sequence of unitaries $v_n\in U(A_n)$ such that $\|\alpha_n(v_n)-v_n\|_2$ does not converge to $0$, then $\alpha$ is not weakly inner. 
\end{lem}
\begin{proof}
Define a (central) sequence 
$$v(n)=1\otimes\dots\otimes1\otimes v_n\otimes1\otimes\dots.$$
We will show that if $\alpha$ is weakly inner then 
$$\|\alpha(v(n))-v(n)\|_2\to0,$$
Now assume there is a unitary $u\in \pi_{\tau}(A)''$ such that $\alpha_g=\Ad u$ on $\pi_{\tau}(A)''$. Then there is a sequence $(x_k)_{k\in\N}$ in $\pi_{\tau}(A)$ such that $x_k\to u$. Let $\epsilon>0$, fix $k$ so that in $\pi(A\rtimes\Z)$ we have $\|u-x_k\|_{2,A\rtimes\Z}\approx_{\epsilon/4}0$ by way of $x_k$ strongly converging to $u$, and let $n$ be large enough so that $[x_k,v(n)]\approx_{\epsilon/2}0$, which is possible because $v(n)$ is a central sequence. We now calculate:
$$\begin{aligned}\|\alpha^{\otimes\N}_g(v(n))-v(n)\|_{2,A}&=\|uv(n)u^*-v(n)\|_{2,A\rtimes\Z}\\
											      &= \|uv(n)-v(n)u\|_{2,A\rtimes\Z}\\	
	 (\|ab\|_2\leq\|a\|_2\|b\|)\quad 				 &\leq 2\|u-x_k\|_{2,A\rtimes\Z}+\|x_kv(n)-v(n)x_k\|\\
	 			 &\eqe0.\end{aligned}$$
 \end{proof}
\bt[Kishimoto]\label{sotrp} Let $A$ be a UHF algebra and $\alpha\in\Aut(A)$. Then $\alpha$ has the tracial Rokhlin property if and only if $\alpha^k$ is strongly outer for all $k>0$.
\et
\bp Combine \cite[Theorem 1.3]{KishZ} and \cite[Theorem 4.5]{kish3}.
\ep

\subsection{Crossed products by outer actions}

\bprop\label{ussn} Let $G$ be a countable discrete amenable group, let $A$ be a UHF algebra and let $\alpha$ be an action of $G$ on $A$. Then we have
\bi
\item[(i)] if $\alpha$ is outer, then  $A\rtimes_{\alpha}G$ is nuclear and simple,\\
\item[(ii)] if $\alpha$ is strongly outer, then  $A\rtimes_{\alpha}G$ has a unique tracial state,\\
\item[(iii)] if $\alpha$ is strongly outer and $G$ is elementary amenable, then  $A\rtimes_{\alpha}G$ is $\Zz$-stable and has real rank zero.\\
\ei
\eprop
\bp 
\bi
\item[(i)] Amenable is \cite{JR}. Simplicity is \cite[Theorem 3.1]{Kish1}.\\
\item[(ii)]  The proof of \cite[Lemma 4.3]{kish3} can be adapted here since strongly outer is the same as uniformly outer \cite[Theorem 4.5]{kish3}.\\
 \item[(iii)] \cite[Theorem 4.9]{MS2} shows $\Zz$-stable. Given $\Zz$-stability by \cite[Corollary 7.3(ii)] {Rordam1} for real rank zero it suffices to show the image of  $K_0(A\rtimes_{\alpha}G)$ under the trace is dense in $\R$ but this contains the image of $K_0(A)$, which is dense.\ei
\ep

\section{``Bump up''s and ``Cut down''s}
\bd Let $A$ be a UHF algebra, $\alpha\in\Aut A$ and let $1\leq k< \infty$. We say $\alpha$ has the order $k$ tracial Rokhlin property, if 
 for every $\epsilon>0$ and every finite subset $\{a_1,\dots, a_n\}$ in $A$, there exist mutually orthogonal projections $p_1,p_2,\dots ,p_{k}$ such that 
\bi
\item $[p_i,a_j]\eqe0$ for $1\leq i\leq k$ and $1\leq j\leq n$,\\
\item $\alpha(p_i)\eqe p_{i+1}$ for $1\leq i\leq k$ with $p_{k+1}=p_1$,\\
\item $\tau(p)\eqe 1$.
\ei
If this is true with $p=1$, we say that $\alpha$ has the order $k$ Rokhlin property. We define the order $\infty$ (tracial) Rokhlin property to mean the (tracial) Rokhlin property.
 \ed

\bd  Let $G$ be an elementary amenable group and let $A$ be a UHF algebra. Suppose each $g\in G$ has order $k(g)$ and $\alpha$ is an action of $G$ on $A$. We say that $\alpha$ has the (tracial) Rokhlin property, if $\alpha_g$ has the order $k(g)$ (tracial) Rokhlin property for all $g\in G$.
\ed
\bl\label{ap} Suppose $\delta\in (0,1/4)$ and $a\in A$ with $a=a^*$ and $a^2\eqd{}a.$
Then there is a projection $p$ such that $p\approx_{2\delta}a.$
\el
\bp This is well-known.
\ep
\bl\label{orthoP} Let $\epsilon>0$ and let $n\in\N$. There is a $\delta>0$ such that for any projections $q_1,\dots,q_n$ such that $q_iq_j\eqd{} 0$,
 there exist mutually orthogonal projections $p_1,\dots,p_n$ such that $p_i\eqe q_i$ for $1\leq i\leq n$.
 
\el
\bp This is well-known.
\ep

\bl\label{Ntrp} Suppose $\alpha\in \Aut M_{(n_l)_{l\in\N}}$ preserves the given decomposition and has the tracial Rokhlin property. Then for any $n'\in\N$ there is $N'>n'$ such that for any finite subset $\{a_1,\dots a_n\}$ and any $\epsilon>0$ there exists $j\geq1$ with $N=n_1n_2\dots n_j$ and mutually orthogonal projections $p_1,\dots, p_{N'}$ in $M_N$ such that
\bi
\item $[p_i,a_j]\eqe0$ for $1\leq i\leq N'$, $1\leq j\leq n$, \\
\item $\alpha(p_i)\eqe p_{i+1}$ for $1\leq i\leq N'-1$,\\
\item $\tau(p_1+\dots+p_{N'})\eqe1$.
\ei
If $\alpha$ has the order $k$ tracial Rokhlin property, ignore $n'$ and require $N'=k$ and $\alpha(p_{k})\eqe p_1$.
\el
\bp Let $n'\in\N$ and fix $N'>n'$ from $\alpha$ having the tracial Rokhlin property. Let $\epsilon>0$ and let $\{a_1,\dots a_n\}$ be a finite subset of $M_{(n_l)_{l\in\N}}$. Without loss of generality assume $\|a_j\|\leq 1$ for $1\leq j\leq n$ and let $\delta=\delta(\epsilon/100,N')$ as in Lemma \ref{ap} such that $\delta\leq \epsilon/100\leq1$.
Since $\alpha$ has the tracial Rokhlin property for $N'>n'$, there exist $q_1,\dots, q_{N'}$ mutually orthogonal projections such that
\bi
\item $[q_i,a_j]\approx_{\delta}0$ for $1\leq i\leq N'$ and $1\leq j\leq n$, \\
\item $\alpha(q_i)\approx_{\delta} q_{i+1}$ for $1\leq i\leq N'-1$,\\
\item $\tau(q_1+\dots+q_{N'})\approx_{\delta}1$.
\ei
By the direct limit definition, there exists $j\geq1$ such that with $N=n_1n_2\dots n_j$, there are self-adjoint $q_1',\dots, q_{N'}'\in M_N\subset M_{(n_l)_{l\in\N}}$ such that 
$$q_i\approx_{\delta/32} q_i'$$
for $1\leq i\leq N'$.
We check $q_i'\approx_{\frac{\delta}{32}(3+\frac{\delta}{32})}(q_i')^2$ and apply Lemma \ref{ap} to get projections $p_i'$ such that 
$$q_i'\approx_{\delta/4} p_i'$$
for $1\leq i\leq N'$. We then check $p_i'p_j'\approx_{\delta}0$ and apply Lemma \ref{orthoP} to mutually orthogonalise the projections $p_i'$ in $M_N$ to obtain $p_1,\dots, p_{N'}\in M_N$ so that 
$$p_i\approx_{\epsilon/100} p_i'.$$ 
It is now routine to check that $p_1,\dots, p_{n'}$ satisfy our claim. See \cite{thesis}.

\ep
\bl\label{ltrp} Suppose $\alpha\in \Aut M_{(n_l)_{l\in\N}}$ preserves the given decomposition and has the tracial Rokhlin property. Then for any $n'\in\N$, any finite subset $\{a_1,\dots a_n\}$ and any $\epsilon>0$ there exists $j\geq1$ such that with $N=n_1n_2\dots n_j$, there exist mutually orthogonal projections $p_1,\dots, p_{n'}$ in $M_N$ such that
\bi
\item $[p_i,a_j]\eqe0$ for $1\leq i\leq n'$ and $1\leq j\leq n$, \\
\item $\alpha(p_i)\eqe p_{i+1}$ for $1\leq i\leq n'-1$,\\
\item $\tau(p_1+\dots+p_{n'})\eqe1$.
\ei
If $\alpha$ has the order $k$ tracial Rokhlin property, ignore $n'$ and require $N'=k$ and $\alpha(p_{k})\eqe p_1$.
\el
\bp Let $n'\in\N$, let $\{a_1,\dots a_n\}$ be a finite subset and let $\epsilon>0$.
Use Lemma \ref{Ntrp} with some $N'$ so large that $\frac{n'}{N'}<\frac{\epsilon}{2}$ and writing $N'=n'Q'+r'$ with $0\leq r'<n'$ and $Q'\in\N$ gives $r'$ satisifying $r'/N'<\epsilon/2$ and get projections $q_1,\dots,q_{N'}$ relative to $\{a_1,\dots a_n\}$ and $\epsilon/N'$. Then group the first $n'Q'$ projections into $n'$ groups as follows: Let $r(j)$ be the remainder when $j$ is divided by $Q'$ and for $1\leq i\leq n'-1$ set
$$p_i=\sum_{r(j)=i\atop 1\leq j\leq n'Q'}q_j$$
and
$$p_{n'}=\sum_{r(j)=0\atop 1\leq j\leq n'Q'}q_j.$$
\ep

\bl\label{Lie} Let $\epsilon>0$, let $N,N'\in\N$, let $\{e_{i,j}\,|\,1\leq i,j \leq N\}$ be a set of matrix units for $M_N$ and let $x\in M_N\otimes M_{N'}$ satisfy 
$$[x,e_{i,j}\otimes1_{N'}]\eqe0$$
 for  $1\leq i,j\leq N$. Then there exists $b\in M_{N'}$ such that
$$x\approx_{10N^3\epsilon} 1\otimes b.$$
\el
\bp This is straightforward.
\ep

\bprop\label{Rokfunction} Suppose $\alpha\in \Aut M_{(n_l)_{l\in\N}}$ preserves the given decomposition. Then $\alpha$ has the tracial Rokhlin property if and only if there is a regrouping $(N_l)_{l\in\N}$ of $(n_l)_{l\in\N}$ such that $M_{N_l}$ contains mutually orthogonal projections $p_1,\dots, p_{l}$ satisfying
\bi
\item $\alpha(p_i) \approx_{\frac{1}{2^l}} p_{i+1}$ for $i\leq l-1$,\\
\item $\tau(p_1+\dots+p_{l})\approx_{\frac{1}{2^l}}1$.
\ei We can also replace $\frac{1}{2^l}$ with any $\epsilon_l$ such that $l\epsilon_l\to0$.
\eprop
\bp 
Assume $\alpha$ has the tracial Rokhlin property. We define the integers $N_l$ inductively on $l$. To get $N_1$, we apply Lemma \ref{ltrp} with $n=0$, $\epsilon=1/2$ and $n'=1$. 

Now assume for an induction that we have found $N_1,\dots, N_l$ with the properties required. Let $N=N_1\cdots N_l$, let $\epsilon=1/2^{l+5}$ and let $\{e_{i,j}\,|\,i,j\leq N\}$ be matrix units for $M_N$. By Lemma \ref{Ntrp} there exists $s>0$ such that with $N_{l+1}=n_{j+1}\dots n_{j+s}$, there are $l+1$ mutually orthogonal projections $q_1,\dots, q_{l+1}$
in $M_N\otimes M_{N_{l+1}}$ such that 
\bi
\item $[q_i,e_{i,j}]\approx_{\frac{\epsilon}{10N^3}}0$ for all $i,j\leq N$,\\
\item $\alpha(q_i) \approx_{\epsilon} q_{i+1}$ for $i\leq l$,\\
\item $\tau(q_1+\dots+q_{l+1})\approx_{\epsilon}1$.
\ei
Now by Lemma \ref{Lie} we have $q_i\eqe 1\otimes b_i$ for some self-adjoint $b_i\in M_{N_{l+1}}$. Then $b_i^2\approx_{4\epsilon}b_i$, so Lemma \ref{ap} gives a projection $p_i\in M_{N_{l+1}}$ such that $p_i\approx_{8\epsilon} b_i$. We check that these projections satisfy our requirements. 


Conversely, let $n'\in\N$ and fix $N'>n'$. Let $\epsilon>0$ and let $\{a_1,\dots, a_n\}$ be a finite subset of $M_{(N_l)_{l\in\N}}$. By the direct limit decomposition there exists $L>0$ such that with $N=N_1\dots N_L$, there exist $a_1(N),\dots, a_n(N)\in M_N$ such that $a_j\approx_{\epsilon/2} a_j(N)$ for $1\leq j\leq n$. Now choose $L'\in\N$ 
\bi
\item $L'>L$\\
\item $N'/L'<\epsilon/2$\\
\item $L'/2^{L'}<\epsilon/2$\\
\ei
Then write $L'$ in quotient remainder form $L'=N'Q'+r'$ with $Q'\in\N$ and $r'\leq N'-1$. If $q_1,\dots, q_{L'}$ are the promised $L'$ mutually orthogonal projections in $M_{L'}$ from the assumption, then define the projections $p_1,\dots, p_{N'}$ as follows: let $r(j)$ be the remainder when $j$ is divided by $N'$. Then for $1\leq i\leq N'-1$ set
$$p_i=\sum_{r(j)=i\atop 1\leq j\leq N'Q'}q_j$$
and 
$$p_{N'}=\sum_{r(j)=0\atop 1\leq j\leq N'Q'}q_j.$$
It can be easily checked these are the required projections.
\ep
\bprop\label{Rokkfunction}
Suppose $\alpha\in \Aut M_{(n_l)_{l\in\N}}$ preserves the given decomposition. Then $\alpha$ has the order $k$ tracial Rokhlin property if and only if there is a regrouping $(N_l)_{l\in\N}$ of $(n_l)_{l\in\N}$ such that $M_{N_l}$ contains mutually orthogonal projections $p_1,\dots, p_{k}$ satisfying
\bi
\item $\alpha(p_i) \approx_{\frac{1}{2^l}} p_{i+1}$ for $i\leq k-1$ and $\alpha(p_k)=p_1$,\\
\item $\tau(p_1+\dots+p_{k})\approx_{\frac{1}{2^l}}1$.
\ei We can also replace $\frac{1}{2^l}$ with any $\epsilon_l$ such that $\epsilon_l\to0$.\qed
\eprop


\bl[Bump-up]\label{bumpup} Suppose $\alpha\in \Aut M_{(s_l)_{l\in\N}}$ preserves the given decomposition and has the tracial Rokhlin property. Then there is a regrouping $(S_l)_{l\in\N}$ of $(s_l)_{l\in\N}$ such that for any $U_l\in U(M_{S_l})$ with
$$\alpha=\bigotimes_{l=1}^{\infty} \Ad U_l$$
and any sequence $(n_l)_{l\in\N}$, there is a regrouping $(N_l)_{\l\in\N}$ of $(n_l)_{l\in\N}$ and integers $Q_l$ and $0\leq r_l<S_l$ for each $l\in\N$ such that
$$\beta=\bigotimes_{n=1}^{\infty} \Ad (\diag(U_l\otimes1_{Q_l},1_{r_l}))$$
has the tracial Rokhlin property on $M_{(N_l)_{l\in\N}}$.
\el
\bp Regroup $(s_l)_{l\in\N}$ using Proposition \ref{Rokfunction} to get a sequence $(S_l)_{l\in\N}$. Now regroup $(n_l)_{l\in\N}$ into $(N_l)_{l\in\N}$ so that 
$$\frac{S_l}{N_l}<\frac{1}{2^l}.$$
Then upon writing $N_l$ in quotient remainder form
$$N_l=Q_lS_l+r_l,$$
for unique $Q_l\in\N$ and $0\leq r_l<S_l$, we have $r_l/N_l<1/2^{l}.$

We now check that $\beta$ has the tracial Rokhlin property. By Proposition \ref{Rokfunction} it suffices to check for each $l\in\N$ that $M_{N_l}$ contains $l$ mutually orthogonal projections $p_1,\dots, p_{l}$ such that for $l\epsilon_l\to0$
\bi
\item $\beta(p_i) \approx_{\epsilon_l} p_{i+1}$ for $1\leq i\leq l-1$,\\
\item $\tau(p_1+\dots+p_{l})\approx_{\epsilon_l}1$.
\ei
Now by construction, we can find $l$ mutually orthogonal projections $p_1, p_2,\dots, p_{l}$ in $M_{S_l}$ such that
\bi
\item $\alpha( p_i )\approx_{1/2^{l}} p_{i+1}$ for $1\leq i\leq l-1$,\\
\item $\tau(p_1+\dots+p_{l})\approx_{1/2^{l}}1$.
\ei
If we regard $M_{S_l}$ as a subalgebra of $M_{N_l}$ via the block diagonal embedding with $Q_l$ copies of $x$ and one zero block of size $r_l$ given by
$$x\mapsto \diag(x\otimes 1_{Q_l},0_{r_l}),$$
then $\beta$ restricts to $\alpha$ on $M_{S_l}$. So $p_1,\dots, p_l$ will work if they have the right trace. Let $\tau_n$ denote the tracial state on $M_{n}$. Then
$$\ba \tau_A(p)&=\tau_{N_l}(p)\\
			&=\tau_{N_l}(1_{S_l})\tau_{S_l}(p)\\
			&=\frac{Q_jS_l}{N_l}\tau_{S_l}(p)\\
			&=\left(1-\frac{r_l}{N_l}\right)\tau_{S_l}(p)\\
			&\approx_{\frac{1}{2^{l}}} 1-\frac{r_l}{N_l}\\
			&\approx_{\frac{1}{2^{l}}} 1.	\ea$$
Since $l/2^{l-1}\to0$, we are done by Proposition \ref{Rokfunction}.
\ep

\bl[Cut-down]\label{cutdown} Let $(n_l)_{l\in\N}$ be a sequence and suppose
$$\alpha=\bigotimes_{l=1}^{\infty} \Ad (\diag(\mu_1^{(l)},\dots,\mu_{n_l}^{(l)}))$$
has the order $k$ tracial Rokhlin property with $\mu_i^{(l)}$ all $k$-th roots of unity. Then there is a regrouping $(N_l)_{l\in\N}$ of $(n_l)_{l\in\N}$ such that
$$\alpha=\bigotimes_{l=1}^{\infty} \Ad (\diag(\lambda_1^{(l)},\dots,\lambda_{N_l}^{(l)}))$$
has $k$ distinct entries in each tensor factor for large $l$.
\el
\bp
By Proposition \ref{Rokkfunction} there is a regrouping $(N_l)_{l\in\N}$ of $(n_l)_{l\in\N}$ so that for each $l\in\N$ there are $k$ non-zero mutually orthogonal projections $p_1,\dots,p_{k}\in M_{N_l}$ such that 
\bi
\item $\alpha(p_i)\approx_{\frac{1}{2^l}}p_{i+1}$ for $1\leq i\leq k$\\
\item $p_{k+1}=p_1$.
\ei
Let $\epsilon>0$. We take $\delta=\delta(\epsilon, k)$ in Lemma \ref{orthoP} and $l$ large enough so that $\frac{1}{2^l}<\delta/10$. Let $\alpha=\Ad u$ on $M_{N_l}$ for $u\in U(M_{N_l})$ and $u^k=1$. Decompose $\C^{N_l}$ into $e^{2\pi ij/k}$-eigenspaces $V_j$ for $u$. By taking a subprojection if necessary we can assume that $p_1$ has rank one. Let $v$ be a unit vector that spans the range of $p_1$. Then we write for unique $v_j\in V_j$:
$$v=\sum_{j=0}^{k-1}v_j.$$
Let $V$ be the subspace spanned by $\{v_j\,|\, 0\leq j\leq k-1\}$. We see that both $p_1$ and $u$ preserve $V$.  Hence also $\alpha^j(p_1)=u^jp_1u^{-j}$ preserves $V$ for all $0\leq j\leq k-1$. Therefore we have $k$ non-zero $\delta$-approximately mutually orthogonal projections in $\End(V)$. By Lemma \ref{orthoP} there exist $k$ non-zero exactly mutually orthogonal projections $\End (V)$, making $V$ a $k$-dimensional space. In particular, $V_j\neq0$ for each $j$ as required.
\ep

\begin{lem}\label{trext}\label{rext}
Let $A$ and $B$ be unital nuclear $C^*$-algebras, let $\alpha\in\Aut A$, let $\beta\in\Aut B$ and consider $\alpha\otimes\beta\in\Aut(A\otimes B)$. (For the claims about the tracial Rokhlin property we assume that every trace on $A\otimes B$ restrict to an extremal trace on either $A$ or $B$, which is true for UHF algebras). We have:
\bi
\item If $\alpha$ satisfies the Rokhlin property (resp.\ tracial Rokhlin property), then $\alpha\otimes\beta$ satisfies the Rokhlin property (resp.\ tracial Rokhlin property).\\
\item If $\alpha$ satisfies the order $k$ Rokhlin property (resp.\ order $k$ tracial Rokhlin property), then $\alpha\otimes\beta$ satisfies the order $k$ Rokhlin property (resp.\ order $k$ tracial Rokhlin property).\\
\item If  $\alpha^k$ has the Rokhlin property (resp.\ tracial Rokhlin property) for some $k>1$ and $\beta$ has the order $k$ Rokhlin property (resp.\ order $k$ tracial Rokhlin property), then $\alpha\otimes\beta$ has the Rokhlin property (resp.\ tracial Rokhlin property). 
\ei 
\el
\begin{proof}The obvious proof works but is lengthy (see \cite{thesis}).     \end{proof}

\bc\label{sotensor} Let $A$ be a UHF algebra and $\alpha,\beta\in\Aut(A)$. Then $\alpha\otimes\beta$ is strongly outer if either $\alpha$ or $\beta$ is strongly outer in all powers.
\ec
\bp Theorem \ref{sotrp} is valid for the order $k$ tracial Rokhlin property and combine this with Lemma \ref{trext}.
\ep
\section{Strongly outer product type actions}
We prove here the existence of strongly outer product type actions.

\bl\label{onQ} Let $G$ be a countable discrete group and suppose that for some sequence $(n_l)_{l\in\N}$ in $\N$ there exists a strongly outer action 
$$\alpha: G{\hookrightarrow}\prod_{l=1}^{\infty}U(M_{n_l})\stackrel{\Ad}{\rightarrow}\Aut M_{(n_l)_{l\in\N}}.$$
Then there exists a strongly outer action
$$\beta: G{\hookrightarrow}\prod_{n=1}^{\infty}U(M_n)\stackrel{\Ad}{\rightarrow}\Aut\Qq.$$
\el
\bp First we can if necessary regroup the tensor factors so that the sequence $(n_l)_{l\in\N}$ is strictly increasing. Next consider the action $\alpha\otimes\id$
$$\alpha\otimes\id:G\to\Aut\left(\bigotimes_{l=1}^{\infty}M_{n_l}\right)\otimes\left(\bigotimes_{l=1}^{\infty}\bigotimes_{n_{l-1}<n<n_l}M_n\right),$$
which is strongly outer by Corollary \ref{sotensor}. Now regroup the factors using Lemma \ref{regroup}.
\ep
So any groups that can act via a strongly outer product type action on some UHF algebra can also be found to act in a standard way on $\Qq$.
We now enlist the help of a Lemma taken from Matui-Sato \cite[Lemma 6.13]{MS2}. Thanks again to Y. Sato for communicating a proof.
\bl\label{socor} Suppose $A$ is a unital simple nuclear $C^*$algebra with a unique tracial state and $\alpha:G\to \Aut(A)$ corresponds to an action of $G$ on $A$ and $\ker\alpha=\{1_G\}$. Then the action $\alpha^{\otimes\N}$ of $G$ on $A^{\otimes\N}=\bigotimes_{n=1}^{\infty}A$ is strongly outer.
\el
\bp We use Lemma \ref{ot} to show that for each $g\in G$ not $1$, the automorphism that $g$ acts by is not weakly inner. Since $\alpha$ is not trivial on $A$ there is some $u\in A$ such that $\alpha(v)\neq v$. We can take $v$ to be unitary since unitaries span $A$ and $\alpha$ is linear. In particular, the sequence $\|\alpha_n(v_n)-v_n\|_2$ in Lemma \ref{ot} with $\alpha_n=\alpha$ and $v_n=v$ for all $n\in\N$ is constant and non-zero.
\ep
\bprop\label{everyouter} Suppose $G$ is a discrete group with a product type action
$$\alpha: G{\hookrightarrow}\prod_{n=1}^{\infty}U(M_{n})\stackrel{\Ad}{\rightarrow}\Aut\Qq,$$
such that $\ker\alpha=\{\id\}$. Then there exists a strongly outer product type action
$$\beta:G{\hookrightarrow}\prod_{n=1}^{\infty}U(M_{n})\stackrel{\Ad}{\rightarrow}\Aut\Qq.$$
\eprop

\bp Consider the action $\alpha^{\otimes\N}$ of $G$ on $\Qq^{\otimes\N}$. It is strongly outer by Lemma \ref{socor} and factorises as
$$\alpha^{\otimes\N}:G\to\prod_{m=1}^{\infty}\prod_{n=1}^{\infty} U(M_n)\to\Aut\left(\bigotimes_{m=1}^{\infty} \bigotimes_{n=1}^{\infty} M_n\right).$$
Now apply Lemma \ref{regroup} and Lemma \ref{onQ}.
\ep
\bc\label{mapso} Let $G$ be a discrete maximally almost periodic group. There exists a strongly outer product type action
$$\beta:G{\hookrightarrow}\prod_{n=1}^{\infty}U(M_{n})\stackrel{\Ad}{\rightarrow}\Aut\Qq.$$
\ec
\bp Combine  Lemma \ref{map} and Proposition \ref{everyouter}.
\ep
We now transfer this to other UHF algebras by ``bumping-up''.

\bl\label{oneg}Suppose $G$ is a discrete maximally almost periodic group and let $g\in G$ have order $1\leq k\leq \infty$. Then for any $(n_l)_{l\in \N}$ there is a regrouping $(N_l)_{l\in\N}$ and a map
$$\beta[g]: G{\hookrightarrow}\prod_{l=1}^{\infty}U(M_{N_l})\stackrel{\Ad}{\rightarrow}\Aut M_{(N_l)_{l\in\N}},$$
such that $\beta[g]_g$ has the order $k$ tracial Rokhlin property. 
\el
\bp Assume $k=\infty$. The other cases are simpler. By Corollary \ref{mapso} we have a strongly outer product type action of $G$ on $\Qq$. In particular, $\alpha_g$ has the tracial Rokhlin property by Theorem \ref{sotrp} and we can apply Lemma \ref{bumpup} for $(n_l)_{l\in\N}$ to get $(N_l)_{l\in\N}$, $(Q_l)_{l\in\N}$, $(r_l)_{l\in\N}$ 
and an automorphism $\beta$ with the tracial Rokhlin property. Since we have a map $\pi_l:G\to U(M_{S_l})$ from being product type we can extend the definition of $\beta$ to an action $\beta[g]$ of $G$ on $M_{(N_l)_{l\in\N}}$ via the maps
$$\beta_l[g]:h \mapsto \Ad(\diag(\pi_l(h)\otimes1_{Q_l},1_{r_l})).$$
Now set $\beta[g]=\otimes_l\beta_l[g]$ and note $\beta[g]_g=\beta$.
\ep

We come to our conclusion:
\bt\label{tRokA} Suppose $G$ is any countable discrete maximally almost periodic group and $A$ is any UHF algebra. Then there exists a strongly outer product type action $\alpha$ of $G$ on $A$. 
\et
\bp Suppose that $A\cong M_{(n_l)_{l\in\N}}$ for some sequence $(n_l)_{l\in\N}$. We can partition this sequence into disjoint subsequences $(n_l(g))_{l\in\N}$ indexed by $g\in G$. Since $G$ is a discrete maximally almost periodic group, by Lemma \ref{oneg}, for each $g\in G$, there is a sequence $(N_l(g))_{l\in\N}$ of the same type as $(n_l(g))_{l\in\N}$ and a map
$$\beta[g]: G{\hookrightarrow}\prod_{l=1}^{\infty}U(M_{N_l(g)})\stackrel{\Ad}{\rightarrow}\Aut M_{(N_l(g))_{l\in\N}},$$
such that $\beta[g]_g$ has the order $k(g)$ tracial Rokhlin property. Taking the product over $g\in G$, we get
$$\beta: G{\hookrightarrow}\prod_{g\in G}\prod_{l=1}^{\infty}U(M_{N_l(g)})\stackrel{\Ad}{\rightarrow}\Aut\Big(\bigotimes_{g\in G}\bigotimes_{l=1}M_{N_l(g)}\Big),$$
where $\beta_g$ has the order $k(g)$ tracial Rokhlin property for all $g\in G$ by Lemma \ref{trext} and hence the action is strongly outer. Now rearrange using Lemma \ref{regroup} to make it more obviousy product type. 
\ep
\bc Suppose $G$ is any countable discrete maximally almost periodic elementary amenable group and $A$ is any UHF algebra. Then there exists a product type action $\alpha$ of $G$ on $A$ with the Matui-Sato tracial Rokhlin property. 
\ec
\bp Use \cite[Theorem 3.7]{MS2}, which works for ``Property (Q)'' also. \ep

\section{Examples of strongly outer product type actions}

We give here a sufficient condition for a tensor product automorphism to be strongly outer in terms of only its trace. We then use it to exhibit some examples of abelian group actions on $\Qq$. In this section only, we adopt the notation  $[u,v]=uvu^*v^*$ for unitaries $u,v\in A$.
\bprop\label{qo}Suppose there are sequences of unitaries such that $\tau([u_n,v_n])$ does not converge to $1$. Then
$\alpha=\bigotimes \Ad u_n$ and $\beta=\bigotimes\Ad v_n$ are both strongly outer as automorphisms of $\Qq$.
\eprop
\bp
Applying Lemma \ref{ot} with $A_n=M_n$ and $u(n)$ the central sequence therein, we check
$$\ba\|\beta(u(n))-u(n)\|_2^2&=\tau_n((v_nu_nv_n^*-u_n)^*(v_nu_nv_n^*-u_n))\\
														&=\tau_n(1-u_n^*v_nu_nv_n^*-v_nu_n^*v_n^*u_n +1)\\
														&=2-\tau_n(u_n^*v_nu_nv_n^*)-\tau_n(v_nu_n^*v_n^*u_n)\\
														&=2-\tau_n(v_nu_nv_n^*u_n^*)-\tau_n(u_nv_nu_n^*v_n^*)\\
														&=2-\tau_n([u_n,v_n]^*)-\tau_n([u_n,v_n])\\
														&=2[1-\text{Re}\,\tau_n([u_n,v_n])].\ea$$
Note now that if $(w_n)_{n\in\N}$ is a sequence of unitaries and $\tau(w_n)\nrightarrow1$,  then Re $\tau(w_n)\nrightarrow1$ (for explicit proof see \cite{thesis}).														
Hence we see that $\beta$ is strongly outer. Now $\tau([u_n,v_n])=\tau([v_n,u_n^*])$ implies that $\alpha^{-1}$ is strongly outer and hence $\alpha$ is strongly outer. \ep

\bl\label{divgp} There are strongly outer actions
$$\R{\hookrightarrow}\prod_{n=1}^{\infty}U(M_n)\stackrel{\Ad}{\rightarrow}\Aut\Qq$$
and
$$\R/\Z{\hookrightarrow}\prod_{n=1}^{\infty}U(M_n)\stackrel{\Ad}{\rightarrow}\Aut\Qq.$$
\el
\bp  Let $\theta$ be a real number and let
$$\theta_n=\begin{cases}&1\quad\text{if $n$ is odd}\\
					&\theta\quad\text{if $n$ is even}.\end{cases}$$
For each $r\in\R$ and $n\in\N$, define a homomorphism $\R\to U(M_n)$ by
$$v_n(r)=\diag(e^{2\pi\theta_n i lr})_{l=1}^n,$$

We again use the condition in Proposition \ref{qo} to check that it is strongly outer.
Let $u_n$ be the permutation matrix corresponding to the cycle $(123\dots n)$. Then (for calculation see \cite{thesis})
$$\tau([u_n,v_n(r)])=e^{-2\pi\theta_n i r}(n^{-1}e^{-2n\pi i r}+n^{-1}(n-1)).$$
Now two limit points of this sequence, corresponding to odd and even $n$, are $ e^{-2\pi\theta i r}$ and $e^{-2\pi i r}$. We see if we chose $\theta$ to irrational, then the limit is $1$ only if $r=0$. If $\theta=1$, then the limit is $1$ only if $r\in\Z$.		 
\ep
Restricting these to $\Q$ we get:
\bc\label{Qso} There exist strongly outer product type actions of $\Q$ and $\Q/\Z$ on $\Qq$.  \qed
\ec

\bprop\label{direct sum} If for each $j\in \N$, $G_j$ has a strongly outer product type action $\alpha_j$ of $G_j$ on $\Qq$, then there exists a strongly outer product type action $\alpha$ of the infinite direct sum $G=\bigoplus_{j\in\N}G_j$ on $\Qq$.
\eprop 
\bp Take $\alpha=\bigotimes_j\alpha_j$. This is strongly outer by Lemma $\alpha\otimes\id$.
\ep
\bc\label{abelian} Every countable discrete abelian group $G$ has a strongly outer product type action on $\Qq$.
\ec
\bp Combine Corollary \ref{Qso} and Propositions \ref{direct sum} and \ref{everyab}.
\ep

\section{Product type actions on $\Qq$ with the Rokhlin property}
We continue our investigation in a related setting.
            
\begin{prop}\label{ext}Let $G$ be a discrete group with a normal subgroup $N$ and let $q:G\to G/N$ be the quotient map. Suppose $A$ and $B$ are unital nuclear $C^*$-algebras, $\alpha$ is an action of $G$ on $A$ such that $\alpha|_N$ has the pointwise Rokhlin property, and $\beta$ is an action of $G/N$ on $B$ that has the pointwise Rokhlin property. Then the action $\gamma$ of $G$ on $A\otimes B$, defined by $\gamma=\alpha\otimes(\beta\circ q)$ has the pointwise Rokhlin property.

\end{prop}
\begin{proof} Note that if $g\in N$, then $\alpha_g$ is an automorphism with the Rokhlin property and if $g\notin N$, then $(\beta\circ q)_g$ is an automorphism with the Rokhlin property. We will now try to apply Lemma \ref{rext} which involves considering the order of $\gamma_g$.
\bi
\item Suppose $\gamma_g$ has infinite order and $(\beta\circ q)_g$ has infinite order, then $g\notin N$ meaning that $(\beta\circ q)_g$ has the Rokhlin property and we are done by applying Lemma \ref{rext} with the roles of $A$ and $B$ reversed. \\
\item If $(\beta\circ q)_g$ has finite order $k$ and $\alpha_g$ has infinite order, then $g^k\in N$ and so $\alpha_g^k=\alpha_{g^k}$ has the Rokhlin property and $(\beta\circ q)_g$ has the Rokhlin property. This case is the content of Lemma \ref{rext}.\\
\item If $\gamma_g$ has finite order, then the orders of $\alpha_g$ and $(\beta\circ q)_g$ are also finite. Let $k$ be the order of $(\beta\circ q)_g$. Then once again, $g^k\in N$ and $\alpha_g^k$ has the Rokhlin property, so Lemma \ref{rext} applies.
\ei
\end{proof}

\begin{lem}\label{fext}
Suppose $G$ is a countable discrete group, $H$ is a subgroup of $G$ with finite index $k$ and let
$$N=\bigcap_{g\in G}gHg^{-1}.$$
Also suppose that $A$ is a UHF algebra and there is a product type action $\alpha_H$ of $H$ on $A$ with the Rokhlin property. Then there exists an action $\alpha_H^G$ of $G$ on $A\otimes M_{k^{\infty}}$ whose restriction to $N$ has the Rokhlin property.
\end{lem}
\begin{proof}  Apply Lemma \ref{induction} to each tensor factor of $\alpha_H$ to get an action $\alpha_H^G$ of $G$ on $A\otimes M_{k^{\infty}}$ such that $\alpha_H^G|_N$ is conjugate to $\alpha_H|_N\otimes \id_{M_{k^\infty}}$, which has the Rokhlin property by Lemma \ref{rext}.
\ep

\begin{thm}\label{index} Suppose $G$ is a countable discrete group with a subgroup $H$ of finite index and normal subgroup of index $k$ given by
$$N=\bigcap_{g\in G}gHg^{-1}.$$ 
Suppose also that $A$ is any UHF algebra. If $H$ has a product type action on $A$ with the Rokhlin property, then there is a product type action of $G$ on $A\otimes M_{k^{\infty}}$ with the Rokhlin property.
\end{thm}
\bp
Let $q$ be the quotient map by $N$. By assumption, there exists a product type action $\alpha_H$ of $H$ on $A$ has the Rokhlin property. Let $j$ be the index of $H$ which divides $k$. We then have by Proposition \ref{fext} an action $\alpha_H^G$
of $G$ on $A\otimes M_{j^{\infty}}$ whose restriction to $N$ has the Rokhlin property. Now since $G/N$ is a finite group of size $k$, Example \ref{finite} shows there exists a product type action $\beta$ of $G/N$ on $M_{k^{\infty}}$ with the Rokhlin property. Combine these using Proposition \ref{ext} to get an action
$\alpha_H^G\otimes (\beta\circ q)$ of $G$ on $A\otimes M_{k^{\infty}}$ with the Rokhlin property.
\end{proof}

Here we see the cut-down principle in action.
\bl\label{onefg}Suppose $G$ is a countable discrete abelian group and let $g\in G$ with finite order $k$. Then there is a map 
$$\alpha[g]: G{\hookrightarrow}\prod_{l=1}^{\infty}U(M_{k}){\hookrightarrow}\Aut M_{k^{\infty}}$$
such that $\alpha[g]_g$ has the order $k$ Rokhlin property. 
\el
\bp 
By Corollary \ref{abelian} there is a strongly outer action $\alpha$ of $G$ on $\Qq$. Hence $\alpha_g$ has the order $k$ tracial Rokhlin property. Let $(N_l)_{l\in\N}$ define a regrouping with respect to $\alpha_g$ as in Lemma \ref{cutdown}. Now since $G$ is abelian, we can diagonalise the image of $G$ in $U(M_{N_l})$ and assume the first $k$ entries correspond to $\beta$ from Lemma \ref{cutdown}. Then restricting to those entries for every $l\in\N$ gives the required action on $M_{k^\infty}$.
\ep
\bt\label{RokQ} Let $\Qq$ be the universal UHF algebra and let $G$ be any countable discrete almost abelian group. Then there exists a product type action
$$G\hookrightarrow\prod_{n=1}^{\infty} U(M_{n})\to\Aut \Qq$$
of $G$ on $\Qq$ with the Rokhlin property.
\et
\bp By Theorem \ref{index} it suffices to assume that $G$ is abelian. In this case combining Corollary \ref{abelian} with Lemma \ref{onefg} we have for each $g\in G$ of finite order $k(g)$, an action $\alpha[g]$ of $G$ on $M_{k(g)^{\infty}}$ such that $\alpha[g]_g$ has the order $k(g)$ Rokhlin property. For $g\in G$ with infinite order, let $\alpha[g]$ be any strongly outer product type action given by Lemma \ref{abelian} and $\alpha[g]_g$ will have the Rokhlin property by Kishimoto \cite[Theorem 1.4]{KishZ}. Taking the tensor product of these over $g\in G$ we get a product type action $\alpha$ of $G$ on $\Qq$ with the Rokhlin property. Finally apply Lemma \ref{regroup} to get a product type action of the form required.
\ep
The above can be made to depend on only one of Sections 3 and 4.

\section{The crossed products $A\rtimes_{\alpha} G$}

\begin{thm}\label{cptrz}Suppose $G$ is a countable discrete maximally almost periodic elementary amenable group, $A$ is any UHF algebra and $\alpha$ is a strongly outer product type action of $G$ on $A$. Then $A\rtimes_{\alpha}G$ is unital simple separable nuclear with tracial rank zero and satisfies the Universal Coefficient Theorem (UCT). Moreover, $A\rtimes_{\alpha}G$ has a unique tracial state.
\end{thm}
\begin{proof} By Proposition \ref{ussn} (i)-(iii), $A\rtimes_{\alpha}G$ is simple nuclear  $\Zz$-stable of real rank zero and has a unique tracial state. The UCT is from Corollary \ref{cpqduct}. By \cite[Theorem 2.1]{Wintr0} it just suffices to show $A\rtimes_{\alpha}G$ has finite decomposition rank. Now by \cite[Corollary 1.2]{MS3}, it suffices to show that $A\rtimes_{\alpha}G$ is quasidiagonal, which is Corollary \ref{cpqduct}. \ep
\begin{cor}Suppose $G$ is a countable discrete almost abelian group and $\alpha$ is a strongly outer product type action of $G$ on the universal UHF algebra $\Qq$. Then $\Qq\rtimes_{\alpha}G$ is unital simple separable nuclear with tracial rank zero and satisfies the Universal Coefficient Theorem. Moreover, $\Qq\rtimes_{\alpha}G$ has a unique tracial state and is locally type I.
\end{cor}
\bp By Lemma \ref{almostn}, $G$ is elementary amenable. By our construction of product type actions we see they are also maximally almost periodic. Hence Theorem \ref{cptrz} applies. For locally type I combine Theorem \ref{productcp} with Corollary \ref{cpqduct}.
\ep

\end{document}